\documentclass[11pt,a4paper]{article}
\usepackage[latin1]{inputenc}
\usepackage{amsmath}
\usepackage{amsthm}
\usepackage{amsfonts}
\usepackage{amsfonts,amsthm,latexsym,amsmath,amssymb,amscd,epsfig,psfrag,enumerate}
\usepackage{graphics,graphicx, bezier, float, color, hyperref}
\usepackage{amssymb,url}
\usepackage{multienum}
\usepackage[table]{xcolor}
\usepackage{multicol,multirow}
\usepackage{graphicx}
\usepackage{fancyvrb}
\usepackage{parskip}
\usepackage[toc,page]{appendix}
\usepackage[top=2.7cm, bottom=2.7cm, left=1.5cm, right=1.5cm]{geometry}
\usepackage{xcolor}

\usepackage{blkarray}
\newtheorem{theorem}{Theorem}[section]
\newtheorem{lemma}{Lemma}[section]

\usepackage[none]{hyphenat}[section]
\newtheorem{definition}{Definition}[section]

\numberwithin{equation}{section}
\numberwithin{table}{section}
\numberwithin{figure}{section}

\title{Multiplicative independence in the sequence of\\ $k-$generalized Lucas numbers}
\author{Herbert Batte$^{1,*} $, Mahadi Ddamulira$^{1}$, Juma Kasozi$^{1}$, and Florian Luca$^{2}$}
\date{}

\begin{document}
\maketitle
\abstract{  Let $ (L_n^{(k)})_{n\geq 2-k} $ be the sequence of $k-$generalized Lucas numbers for some fixed integer $k\ge 2$, whose first $k$ terms are $0,\;\ldots\;,\;0,\;2,\;1$ and each term afterward is the sum of the preceding $k$ terms. In this paper, we find all pairs of the $k-$generalized Lucas numbers that are multiplicatively dependent. } 



\thanks{$ ^{*} $ Corresponding author}

\section{Introduction}\label{intro}
\subsection{Background}
\label{sec:1.1}
Let $k \ge 2$ be an integer. The sequence of $k$-generalized Fibonacci numbers $(F_n^{(k)})_{n\in {\mathbb Z}}$ is the sequence given by 
$$
F_n^{(k)}=F_{n-1}^{(k)}+\cdots+F_{n-k}^{(k)}
$$
with $F_1^{(k)}=1$ and $F_{i}^{(k)}=0$ for $i=2-k,3-k,\ldots,0$. If one keeps the same recurrence relation (every term is the sum of the previous $k$ terms) but instead of the initial values being 
$0,0,\ldots,0,0,1$ one starts with $0,0,\ldots,0,2,1$, then one obtains the sequence of $k-$generalized Lucas numbers. When $k = 2$, these coincide with the classical sequences  of Fibonacci and Lucas numbers. 
If $a,~b$ are nonzero integers then we say that they are multiplicatively independent if the only solution in integers $x,y$ of the equation $a^x b^y=1$ is $x=y=0$ and multiplicatively dependent otherwise. Note that any pair of integers such that one of them is $1$ is a pair of multiplicatively dependent integers. 

In \cite{GoL}, G\'omez and Luca fixed $k\ge 2$ and looked at the instances $n>m\ge 1$ for which $F_n^{(k)}$ and $F_m^{(k)}$ are multiplicatively dependent. This is certainly so if $m\in \{1,2\}$ since 
$F_1^{(k)}=F_2^{(k)}=1$. In addition, since $F_m^{(k)}=2^{m-2}$ for all $m\in [2,k+1]$, it follows that $F_n^{(k)}$ and $F_m^{(k)}$ are multiplicatively dependent when $3\le m<n\le k+1$. The main result of 
\cite{GoL} is that these are the only instances when $F_n^{(k)}$ and $F_m^{(k)}$ are multiplicatively dependent when $k\ge 3$. When $k=2$, there is the additional instance $(m,n)=(3,6)$ as $F_3^{(2)}=2$ and $F_6^{(2)}=2^3$. The instance $k=2$ can be easily treated with Carmichael's primitive divisor theorem \cite{Car} which asserts that $F_n^{(2)}$ has a primitive prime factor (a prime factor $p$ which does not divide 
$F_m^{(2)}$ for any positive integer $m<n$) provided $n>12$. Then an immediate verification of the cases $3\le m\le n\le 12$ gives the solution $(m,n)=(3,6)$.

Similar results concerning the multiplicative dependence between terms of other linear recurrence sequences can be seen in \cite{GGL} and \cite{BR}. In this paper, we investigate the multiplicative dependence between $L_m^{(k)}$ and $L_n^{(k)}$ for $n>m$. We assume again that $n\ne 1,~m\ne 1$ as $L_1^{(k)}=1$.  We prove the following result.
\subsection{Main Result}\label{sec:1.2}
\begin{theorem}\label{1.2a} 
	Let  $k\ge 2$ and $(L_n^{(k)})_{n\in {\mathbb Z}}$ be the sequence of $k$--generalized Lucas numbers.  If $n>m\ge 0$ are integers with $n\ne 1,~m\ne 1$ such that $L_n^{(k)}$ and $L_m^{(k)}$ are multiplicatively dependent then $(k,m,n)=(2,0,3),~(3,0,7)$. 
\end{theorem}

\section{Methods}
\subsection{Preliminaries}
The following is a well known identity
\begin{align}\label{eq2.2}
	L_n^{(k)} = 3 \cdot 2^{n-2},\qquad \text{for all}\qquad 2 \le n \le k.
\end{align}
It can be easily proved by induction. The characteristic polynomial of $(L_n^{(k)})_{n\in {\mathbb Z}}$ is given by
\[
\Psi_k(x) = x^k - x^{k-1} - \cdots - x - 1.
\]
This polynomial is irreducible in $\mathbb{Q}[x]$. The polynomial $\Psi_k(x)$ possesses a unique real root $\alpha:=\alpha(k)>1$, see \cite{MIL}. All the other roots of $\Psi_k(x)$ are inside the unit circle,  see \cite{MILLER}. The particular root $\alpha$ can be found in the interval
\begin{align}\label{eq2.3}
2(1 - 2^{-k} ) < \alpha < 2,\qquad \text{for} \qquad k\ge 2,
\end{align}
as noted in \cite{WOL}. As in the classical case when $k=2$, it was shown in \cite{BRL} that 
\begin{align}\label{eq2.4}
	\alpha^{n-1} \le L_n^{(k)}\le2\alpha^n, \qquad \text{holds for all}\qquad n\ge1, \quad k\ge 2.
\end{align}
For $k\ge 2$ we define
\begin{align*}
	f_k(x):=\dfrac{x-1}{2+(k+1)(x-2)}.
\end{align*}
If $2(1 - 2^{-k} ) <  x < 2$, then $\frac{\partial}{\partial x}f_k(x)<0$ and so inequality \eqref{eq2.3} implies that
\begin{align}\label{eq2.5}
	\dfrac{1}{2}=f_k(2)<f_k(\alpha)<f_k(2(1 - 2^{-k} ))\le \dfrac{3}{4},
\end{align}
holds for all $k\ge 3$. The inequality $1/2=f_k(2)<f_k(\alpha)<3/4$ also holds for $k=2$. In addition, it is easy to verify that $|f_k(\alpha_i)|<1$ holds for all $2\le i\le k$, where $\alpha_i$ for $i=2,\ldots,k$ are the roots of $\Psi_k(x)$ inside the unit disk.
Lastly, it was shown in \cite{BRL} that
\begin{align}\label{eq2.6}
	L_n^{(k)}=\displaystyle\sum_{i=1}^{k}(2\alpha_i-1)f_k(\alpha_i)\alpha_i^{n-1}~~\text{and}~~\left|L_n^{(k)}-f_k(\alpha)(2\alpha-1)\alpha^{n-1}\right|<\dfrac{3}{2},
\end{align}
hold for all $k\ge 2$ and $n\ge 2-k$. This means that $L_n^{(k)}=f_k(\alpha)(2\alpha-1)\alpha^{n-1}+e_k(n)$, where $|e_k(n)|<1.5$. The first expression in \eqref{eq2.6} is known as the Binet-like formula for $L_n^{(k)}$. Furthermore, the second inequality expression in \eqref{eq2.6} shows that the contribution of the zeros that are inside the unit circle to $L_n^{(k)}$ is small. 
\subsection{Linear forms in logarithms}
We use Baker-type lower bounds for nonzero linear forms in logarithms of algebraic numbers. There are many such bounds mentioned in the literature but we use one of Matveev from \cite{MAT}. Before we can formulate such inequalities we need the notion of height of an algebraic number recalled below.  

\begin{definition}\label{def2.1}
	Let $ \gamma $ be an algebraic number of degree $ d $ with minimal primitive polynomial over the integers $$ a_{0}x^{d}+a_{1}x^{d-1}+\cdots+a_{d}=a_{0}\prod_{i=1}^{d}(x-\gamma^{(i)}), $$ where the leading coefficient $ a_{0} $ is positive. Then, the logarithmic height of $ \gamma$ is given by $$ h(\gamma):= \dfrac{1}{d}\Big(\log a_{0}+\sum_{i=1}^{d}\log \max\{|\gamma^{(i)}|,1\} \Big). $$
\end{definition}
In particular, if $ \gamma$ is a rational number represented as $\gamma=p/q$ with coprime integers $p$ and $ q\ge 1$, then $ h(\gamma ) = \log \max\{|p|, q\} $. 
The following properties of the logarithmic height function $ h(\cdot) $ will be used in the rest of the paper without further reference:
\begin{equation}\nonumber
	\begin{aligned}
		h(\gamma_{1}\pm\gamma_{2}) &\leq h(\gamma_{1})+h(\gamma_{2})+\log 2;\\
		h(\gamma_{1}\gamma_{2}^{\pm 1} ) &\leq h(\gamma_{1})+h(\gamma_{2});\\
		h(\gamma^{s}) &= |s|h(\gamma)  \quad {\text{\rm valid for}}\quad s\in \mathbb{Z}.
	\end{aligned}
\end{equation}
With these properties, it was easily computed in Section 3 of \cite{GoL} that
\begin{align}\label{eq2.7}
	h\left(f_k(\alpha)\right)<2\log k, \qquad \text{for all}\qquad k\ge 2.
\end{align}

A linear form in logarithms is an expression of the form
\begin{equation}
	\label{eq:Lambda}
	\Lambda:=b_1\log \gamma_1+\cdots+b_t\log \gamma_t,
\end{equation}
where $\gamma_1,\ldots,\gamma_t$ are positive real  algebraic numbers and $b_1,\ldots,b_t$ are nonzero integers. We assume, $\Lambda\ne 0$. We need lower bounds 
for $|\Lambda|$. We write ${\mathbb K}={\mathbb Q}(\gamma_1,\ldots,\gamma_t)$ and $D$ for the degree of ${\mathbb K}$.
We give the general form due to Matveev \cite{MAT}. 

\begin{theorem}[Matveev, \cite{MAT}]
	\label{thm:Mat} 
	Put $\Gamma:=\gamma_1^{b_1}\cdots \gamma_t^{b_t}-1=e^{\Lambda}-1$. Then 
	$$
	\log |\Gamma|>-1.4\cdot 30^{t+3}\cdot t^{4.5} \cdot D^2 (1+\log D)(1+\log B)A_1\cdots A_t,
	$$
	where $B\ge \max\{|b_1|,\ldots,|b_t|\}$ and $A_i\ge \max\{Dh(\gamma_i),|\log \gamma_i|,0.16\}$ for $i=1,\ldots,t$.
\end{theorem}
In our application of Matveev's result (Theorem \ref{thm:Mat}), we need to ensure that the linear forms in logarithms are indeed nonzero. To ensure this, we shall need the following two results given as Lemma 2.7 and Lemma 2.8 in \cite{GGL}.
\begin{lemma}[Lemma 2.7 in \cite{GGL}]\label{lemGGL}
	The only solution in positive integers \( k \geq 2 \), \( x \), \( y \) of the Diophantine equation
	\[
	(2^{k+1} - 3)^x = \left(\frac{2^{k+1}k^k - (k + 1)^{k+1}}{(k - 1)^2}\right)^y
	\]
	is \( k = 2 \), for which \( x = y \).
	
\end{lemma}
\begin{lemma}[Lemma 2.8 in \cite{GGL}]\label{lemGL}
Let \( N := N_{\mathbb{K}/\mathbb{Q}}, \) where \( \mathbb{K} = \mathbb{Q}(\alpha) \). Then
\begin{enumerate}[(i)]
	\item For \( n, m \geq 1 \) and \( k \geq 2 \), \( |N(\alpha)| = 1 \).
	\item \( N(2\alpha - 1) = 2^{k+1} - 3 \) and \( N(f_k(\alpha)) = (k - 1)^2 / (2^{k+1}k^k - (k + 1)^{k+1}) \).
	\item For \( k \geq 3 \), \( N((2\alpha - 1)f_k(\alpha)) < 1 \).
\end{enumerate}	
\end{lemma}
For the special case when $t=2$ in Theorem \ref{thm:Mat}, we instead apply a sharper result from \cite{LMN}. Suppose that $A_1>1,~A_2>1$ are real numbers such that 
\begin{equation}
	\label{eq:Ai}
	\log A_i\ge \max\left\{h(\gamma_i),\frac{|\log \gamma_i|}{D},\frac{1}{D}\right\}\quad {\text{\rm for}}\quad i=1,2.
\end{equation}
Put
$$
b':=\frac{|b_1|}{D\log A_2}+\frac{|b_2|}{D\log A_1}.
$$
The following  result is Corollary 2 in \cite{LMN}. 
\begin{theorem}[Laurent et al, 1995]
	\label{thm:LMN}
	In case $t=2$, we have  
	$$
	\log |\Lambda|\ge -24.34 D^4\left(\max\left\{\log b'+0.14,\frac{21}{D},\frac{1}{2}\right\}\right)^2\log A_1\log A_2.
	$$
\end{theorem}

However, during the calculations, upper bounds on the variables which are too large are obtained, thus there is need to reduce them. In this paper we use the following result related with continued fractions (see Theorem 8.2.4 in \cite{ME}).

\begin{lemma}[Legendre]\label{lem:Leg} Let $ \tau $ be an irrational number, $[a_0,a_1,a_2,\ldots]$ be the continued fraction expansion of $\tau$. Let $p_i/q_i=[a_0,a_1,a_2,\ldots,a_i]$, for all $i\ge 0$, be all the convergents of the continued fraction of $ \tau$, and $M$ be a positive integer. Let $ N $ be a non-negative integer such that
	$ q_{N} > M $.
	Then putting $ a(M) := \max \{a_{i}: i=0,1,2,\ldots,N   \}$, the inequality
	$$ \bigg| \tau-\frac{r}{s}   \bigg| > \dfrac{1}{(a(M)+2)s^2},  $$
	holds for all pairs $ (r, s) $ of positive integers with $ 0 < s < M $. 
\end{lemma}
However, since there are no methods based on continued fractions to find a lower bound for linear forms in more than two variables with bounded integer coefficients, we use at some point a method based on the LLL--algorithm. We next explain this method.

Let $k$ be a positive integer. A subset $\mathcal{L}$ of the $k$--dimensional real vector space ${ \mathbb{R}^k}$ is called a lattice if there exists a basis $\{b_1, b_2, \ldots, b_k \}$ of $\mathbb{R}^k$ such that
\begin{align*}
	\mathcal{L} = \sum_{i=1}^{k} \mathbb{Z} b_i = \left\{ \sum_{i=1}^{k} r_i b_i \mid r_i \in \mathbb{Z} \right\}.
\end{align*}
We say that $b_1, b_2, \ldots, b_k$ form a basis for $\mathcal{L}$, or that they span $\mathcal{L}$. We
call $k$ the rank of $ \mathcal{L}$. The determinant $\text{det}(\mathcal{L})$, of $\mathcal{L}$ is defined by
\begin{align*}
	\text{det}(\mathcal{L}) = | \det(b_1, b_2, \ldots, b_k) |,
\end{align*}
with the $b_i$'s being written as column vectors. This is a positive real number that does not depend on the choice of the basis (see \cite{Cas}, Section 1.2).

Given linearly independent vectors $b_1, b_2, \ldots, b_k $ in $ \mathbb{R}^k$, we refer back to the Gram--Schmidt orthogonalization technique. This method allows us to inductively define vectors $b^*_i$ (with $1 \leq i \leq k$) and real coefficients $\mu_{i,j}$ (for $1 \leq j \leq i \leq k$). Specifically,
\begin{align*}
	b^*_i &= b_i - \sum_{j=1}^{i-1} \mu_{i,j} b^*_j,~~~
	\mu_{i,j} = \dfrac{\langle b_i, b^*_j\rangle }{\langle b^*_j, b^*_j\rangle},
\end{align*}
where \( \langle \cdot , \cdot \rangle \)  denotes the ordinary inner product on \( \mathbb{R}^k \). Notice that \( b^*_i \) is the orthogonal projection of \( b_i \) on the orthogonal complement of the span of \( b_1, \ldots, b_{i-1} \), and that \( \mathbb{R}b_i \) is orthogonal to the span of \( b^*_1, \ldots, b^*_{i-1} \) for \( 1 \leq i \leq k \). It follows that \( b^*_1, b^*_2, \ldots, b^*_k \) is an orthogonal basis of \( \mathbb{R}^k \). 
\begin{definition}
	The basis $b_1, b_2, \ldots, b_n$ for the lattice $\mathcal{L}$ is called reduced if
	\begin{align*}
		\| \mu_{i,j} \| &\leq \frac{1}{2}, \quad \text{for} \quad 1 \leq j < i \leq n,~~
		\text{and}\\
		\|b^*_{i}+\mu_{i,i-1} b^*_{i-1}\|^2 &\geq \frac{3}{4}\|b^*_{i-1}\|^2, \quad \text{for} \quad 1 < i \leq n,
	\end{align*}
	where $ \| \cdot \| $ denotes the ordinary Euclidean length. The constant $ {3}/{4}$ above is arbitrarily chosen, and may be replaced by any fixed real number $ y $ in the interval ${1}/{4} < y < 1$ (see \cite{LLL}, Section 1).
\end{definition}
Let $\mathcal{L}\subseteq\mathbb{R}^k$ be a $k-$dimensional lattice  with reduced basis $b_1,\ldots,b_k$ and denote by $B$ the matrix with columns $b_1,\ldots,b_k$. 
We define
\[
l\left( \mathcal{L},y\right)= \left\{ \begin{array}{c}
	\min_{x\in \mathcal{L}}||x-y|| \quad  ;~~ y\not\in \mathcal{L}\\
	\min_{0\ne x\in \mathcal{L}}||x|| \quad  ;~~ y\in \mathcal{L}
\end{array}
\right.,
\]
where $||\cdot||$ denotes the Euclidean norm on $\mathbb{R}^k$. It is well known that, by applying the
LLL--algorithm, it is possible to give in polynomial time a lower bound for $l\left( \mathcal{L},y\right)$, namely a positive constant $c_1$ such that $l\left(\mathcal{L},y\right)\ge c_1$ holds (see \cite{SMA}, Section V.4).
\begin{lemma}\label{lem2.5}
	Let $y\in\mathbb{R}^k$ and $z=B^{-1}y$ with $z=(z_1,\ldots,z_k)^T$. Furthermore, 
	\begin{enumerate}[(i)]
		\item if $y\not \in \mathcal{L}$, let $i_0$ be the largest index such that $z_{i_0}\ne 0$ and put $\lambda:=\{z_{i_0}\}$, where $\{\cdot\}$ denotes the distance to the nearest integer.
		\item if $y\in \mathcal{L}$, put $\lambda:=1$.
	\end{enumerate}
Finally, we set 
	\[
	c_1:=\max\limits_{1\le j\le k}\left\{\dfrac{||b_1||}{||b_j^*||}\right\}\qquad\text{and}\qquad	
	 \delta:=\lambda\dfrac{||b_1||}{c_1}.
	\]
\end{lemma}

In our application, we are given real numbers $\eta_0,\eta_1,\ldots,\eta_k$ which are linearly independent over $\mathbb{Q}$ and two positive constants $c_3$ and $c_4$ such that 
\begin{align}\label{2.9}
	|\eta_0+x_1\eta_1+\cdots +x_k \eta_k|\le c_3 \exp(-c_4 H),
\end{align}
where the integers $x_i$ are bounded as $|x_i|\le X_i$ with $X_i$ given upper bounds for $1\le i\le k$. We write $X_0:=\max\limits_{1\le i\le k}\{X_i\}$. The basic idea in such a situation, from \cite{Weg}, is to approximate the linear form \eqref{2.9} by an approximation lattice. So, we consider the lattice $\mathcal{L}$ generated by the columns of the matrix
$$ \mathcal{A}=\begin{pmatrix}
	1 & 0 &\ldots& 0 & 0 \\
	0 & 1 &\ldots& 0 & 0 \\
	\vdots & \vdots &\vdots& \vdots & \vdots \\
	0 & 0 &\ldots& 1 & 0 \\
	\lfloor C\eta_1\rfloor & \lfloor C\eta_2\rfloor&\ldots & \lfloor C\eta_{k-1}\rfloor& \lfloor C\eta_{k} \rfloor
\end{pmatrix} ,$$
where $C$ is a large constant usually of the size of about $X_0^k$ . Let us assume that we have an LLL--reduced basis $b_1,\ldots, b_k$ of $\mathcal{L}$ and that we have a lower bound $l\left(\mathcal{L},y\right)\ge c_1$ with $y:=(0,0,\ldots,-\lfloor C\eta_0\rfloor)$. Note that $ c_1$ can be computed by using the results of Lemma \ref{lem2.5}. Then, with these notations the following result  is Lemma VI.1 in \cite{SMA}.
\begin{lemma}[Lemma VI.1 in \cite{SMA}]\label{lem2.6}
	Let $S:=\displaystyle\sum_{i=1}^{k-1}X_i^2$ and $T:=\dfrac{1+\sum_{i=1}^{k}X_i}{2}$. If $\delta^2\ge T^2+S$, then inequality \eqref{2.9} implies that we either have $x_1=x_2=\cdots=x_{k-1}=0$ and $x_k=-\dfrac{\lfloor C\eta_0 \rfloor}{\lfloor C\eta_k \rfloor}$, or
	\[
	H\le \dfrac{1}{c_4}\left(\log(Cc_3)-\log\left(\sqrt{\delta^2-S}-T\right)\right).
	\]
\end{lemma}
Finally, we present an analytic argument which is Lemma 7 in \cite{GL}. 
\begin{lemma}[Lemma 7 in \cite{GL}]\label{Guz} If $ s \geq 1 $, $T > (4s^2)^s$ and $T > \displaystyle \frac{z}{(\log z)^s}$, then $$z < 2^s T (\log T)^s.$$	
\end{lemma}
SageMath 9.5 is used to perform all computations in this work.

\section{Proof of the main result}
We do this in the following cases.
\subsection{The case $k=2$}
In this subsection we study the Diophantine equation
\begin{equation}\label{eq3.1}
	(L_n)^x = (L_m)^y, 
\end{equation}
where $n>m\ge 0$, $n\ne 1,m\ne 1$ and $x,y\in \mathbb{Z}$. 
\begin{lemma}
The only integer solution $(n, m)$ of equation \eqref{eq3.1} in the range $n>m\ge 0$, $n\ne 1,m\ne 1$  is $(n,m)=(3,0)$.
\end{lemma}
\begin{proof}
By Carmichael's Primitive Divisor Theorem and the known relation $F_{2k} = F_k L_k$, it follows that for $k>6$, $L_k$ has a primitive prime factor (namely, the primitive prime factor of $F_{2k}$). In particular, equation \eqref{eq3.1} is impossible with $x\ne 0$ if $n>6$. One now checks that in the remaining range $0\le m<n\le 6$, the only convenient solution is $(m,n)=(3,0)$. 
\end{proof}

\subsection{The case $k\ge 3$}
The equation is 
\begin{equation}\label{eq3.2}
	(L_n^{(k)})^x = (L_m^{(k)})^y, 
\end{equation}
where none of $x,~y$ is zero. Since $n>m\ge 0$, and $n\ne 1,~m\ne 1$ we have that $n\ge 2$.  If $m\in [0,k]$, $m\ne 1$, we then have that either $m=0$ and $L_0^{(k)}=2$ or $m\ge 2$ and $L_m^{(k)}=3\cdot 2^{m-2}$.
From the above formulas, we get that if $0\le m\le k$, $m\ne 1$, then $n\ge k+1$. Furthermore, $P(L_n^{(k)})\le 3$. All $k$--generalized Lucas numbers $L_n^{(k)}$ with $P(L_n^{(k)})\le 7$ and $n\ge k+1$ have been found in \cite{Bat} and they are the following
 \begin{align*}
	&L_3^{(2)}=2^2,~~~~ L_4^{(2)}=7,~~~~ L_6^{(2)}=2\cdot 3^2,~~~~ L_4^{(3)}=2\cdot 5,~~~~ L_6^{(3)}=5\cdot 7, \\
	&L_7^{(3)}=2^6,~~~~ L_{12}^{(3)}=2\cdot 3^3\cdot 5^2,~~~~ L_{15}^{(3)}=2^4\cdot 3\cdot 5^2\cdot 7,
	~~~~ L_8^{(4)}=2^5\cdot 5 ~~\text{and}~~ L_{15}^{(10)}=2^2\cdot 5^3\cdot 7^2.
\end{align*}
We learn from the above list that the only possibilities are $(k,m,n)=(2,0,3)$ and $(3,0,6)$. From now on, we assume that $P(L_m^{(k)})\ge 11$. In particular, $n>m\ge k+1$. In addition, $L_n^{(k)}>L_m^{(k)}$ showing that 
$x<y$. It is a straightforward exercise to show that relation \eqref{eq2.4} together with \eqref{eq3.2} imply that $x < m$ and $y < n$. Therefore, $n=\max\{n,m,x,y,k\}$.

Assume next that $n\le 50$. Then $4\le n\le 50$, $3\le m\le 49$ and $3\le k\le 48$.  We wrote a code in SageMath to find all coincidences for which equations \eqref{eq3.2} hold. We find none in these ranges, see Appendix 1. From now on, $n>50$.

\subsection{An inequality for $n$ in terms of $k$}
\begin{lemma}\label{lem3.1}
	Let $(n,m,x,y,k)$ be a solution to \eqref{eq3.1} with $n>m\ge k+1$, $k\ge 3$, $n>50$ and $x<y<n$. Then 
	\begin{align*}
		n<8.5 \cdot10^{34}k^8 (\log k)^6.
	\end{align*} 
\end{lemma}
\begin{proof} Here, we first go back to \eqref{eq2.6} and write it as
\begin{align}\label{eq3.5}
	L_n^{(k)}=f_k(\alpha)(2\alpha-1)\alpha^{n-1}+e_k(n), ~~\text{where}~~ |e_k(n)|<1.5.
\end{align}
Hence, 
\begin{align}\nonumber
	\left(L_n^{(k)}\right)^x=\left(f_k(\alpha)\right)^x(2\alpha-1)^x\alpha^{(n-1)x}\left[1+\dfrac{e_k(n)}{f_k(\alpha)(2\alpha-1)\alpha^{n-1}}\right]^x.
\end{align}
Now, we study the elements
\begin{align}\label{eq3.6}
	(1+r)^x~~\text{and}~~z:=xr,~~~~\text{where}~~r=\dfrac{e_k(n)}{f_k(\alpha)(2\alpha-1)\alpha^{n-1}}.
\end{align}
Since $k\ge 3$, then $\alpha\ge 1.75$ by relation \eqref{eq2.3} and so $2\alpha-1\ge 2.5$. Moreover, $f_k(\alpha)>{1}/{2}$ by inequality \eqref{eq2.5}. This means that 
\begin{align*}
	|r|=\left|\dfrac{e_k(n)}{f_k(\alpha)(2\alpha-1)\alpha^{n-1}}\right|<\dfrac{1.5}{0.5\cdot 2.5\cdot 1.5^{n-1}}=\dfrac{1.2}{1.5^{n-1}}<\dfrac{1}{1.5^{n-2}},
\end{align*}
and 
\begin{align*}
	|z|=x|r|<\dfrac{n}{1.5^{n-2}}.
\end{align*}
In particular, $|z|<10^{-6}$ for all $n>50$. Now, if $r<0$, then 
\begin{align*}
1>(1+r)^x = \exp \left(x\log(1-|r|)\right)	\ge \exp(-2|z|)>1-2|z|,
\end{align*} 
and if $r>0$, then
\begin{align*}
	1<(1+r)^x=\left(1+\dfrac{|z|}{x}\right)^x<\exp(|z|)<1+2|z|,
\end{align*}
because $|r|<0.5$ for all $n>50$ and $|z|<10^{-6}$ is very small. Therefore in both cases we have that 
\begin{align}\label{eq3.7}
	\left|\left(L_n^{(k)}\right)^x-\left(f_k(\alpha)\right)^x(2\alpha-1)^x\alpha^{(n-1)x}\right|&<2|z|\left(f_k(\alpha)\right)^x(2\alpha-1)^x\alpha^{(n-1)x}\nonumber\\
	&< \dfrac{2n\left(f_k(\alpha)\right)^x(2\alpha-1)^x\alpha^{(n-1)x}}{1.5^{n-2}}.
\end{align}
We now use equation \eqref{eq3.2} together with \eqref{eq3.7} to have
	\begin{align}\label{eq3.8}
		&\left|\left(f_k(\alpha)\right)^x(2\alpha-1)^x\alpha^{(n-1)x}-\left(f_k(\alpha)\right)^y(2\alpha-1)^y\alpha^{(m-1)y}\right| \nonumber\\
		&=\left|\left(f_k(\alpha)\right)^x(2\alpha-1)^x\alpha^{(n-1)x}-\left(f_k(\alpha)\right)^y(2\alpha-1)^y\alpha^{(m-1)y}-\left(L_n^{(k)}\right)^x + \left(L_m^{(k)}\right)^y\right|\nonumber\\
		&<\left|\left(L_n^{(k)}\right)^x-\left(f_k(\alpha)\right)^x(2\alpha-1)^x\alpha^{(n-1)x}\right|+\left|\left(L_m^{(k)}\right)^y-\left(f_k(\alpha)\right)^y(2\alpha-1)^y\alpha^{(m-1)y}\right|\nonumber\\
		&<\dfrac{2n\left(f_k(\alpha)\right)^x(2\alpha-1)^x\alpha^{(n-1)x}}{1.5^{n-2}}+\dfrac{2m\left(f_k(\alpha)\right)^y(2\alpha-1)^y\alpha^{(m-1)y}}{1.5^{m-2}}\nonumber\\
		&<\dfrac{4m}{1.5^{m-2}}\left(\left(f_k(\alpha)\right)^x(2\alpha-1)^x\alpha^{(n-1)x}+\left(f_k(\alpha)\right)^y(2\alpha-1)^y\alpha^{(m-1)y}\right),
	\end{align}
where in the last inequality we have used the fact that $x\mapsto x/1.5^x$ is decreasing for all $x\ge 3$, together with $n>m\ge k+1\ge 4$. 

Dividing both sides of \eqref{eq3.8} by $\max\left\{\left(f_k(\alpha)\right)^x(2\alpha-1)^x\alpha^{(n-1)x},~\left(f_k(\alpha)\right)^y(2\alpha-1)^y\alpha^{(m-1)y}\right\}$, we get
\begin{align}\label{eq3.9}
&\left|1-\left(f_k(\alpha)\right)^{\Delta(y-x)}(2\alpha-1)^{\Delta(y-x)}\alpha^{\Delta((m-1)y-(n-1)x)}\right|\nonumber\\
&<\dfrac{4m}{1.5^{m-2}}\left(1+\left(f_k(\alpha)\right)^{\Delta(y-x)}(2\alpha-1)^{\Delta(y-x)}\alpha^{\Delta((m-1)y-(n-1)x)}\right)< \dfrac{4m}{1.5^{m-2}}(1+1)<\dfrac{19m}{1.5^{m}}	,
\end{align}
where $\Delta\in \{\pm 1\}$. 
We now intend to apply Theorem \ref{thm:Mat} on the left--hand side of \eqref{eq3.9}. Let $$\Gamma:=\left(f_k(\alpha)\right)^{\Delta(y-x)}(2\alpha-1)^{\Delta(y-x)}\alpha^{\Delta((m-1)y-(n-1)x)}-1=e^{\Lambda}-1.$$
If $\Lambda=0$, then 
\begin{align*}
	\left(f_k(\alpha)\right)^{y-x}=(2\alpha-1)^{x-y}\alpha^{(n-1)x-(m-1)y}.
\end{align*}
Applying norms in ${\mathbb K}={\mathbb Q}(\alpha)$ and using $|N(\alpha)|=1$, and item (i) of Lemma \ref{lemGL}, the above equation becomes
\begin{align*}
	|N\left(f_k(\alpha)\right)|^{y-x}=|N(2\alpha-1)|^{x-y},
\end{align*}
which implies that
\begin{align*}
	\dfrac{1}{|N\left(f_k(\alpha)\right)|}=|N(2\alpha-1)|.
\end{align*}
Here, we used that $y>x$ (in particular, $y-x\ne 0$). Now, by item (ii) of Lemma \ref{lemGL}, we have from the above that
	\[
 \frac{2^{k+1}k^k - (k + 1)^{k+1}}{(k - 1)^2}=2^{k+1} - 3,
\]
which gives $k=2$, by Lemma \ref{lemGGL}. This contradicts the working assumption that $k\ge 3$. Thus, $\Lambda\ne 0$. We use the field $\mathbb{Q}(\alpha)$ with
degree $D = k$. Here, $t = 3$,
\begin{equation}\nonumber
	\begin{aligned}
		\gamma_{1}&:=f_k(\alpha), ~~~~~~\gamma_{2}:=2\alpha-1, ~~~~~\gamma_{3}:=\alpha,\\
		b_1&:=\Delta(y-x), ~~b_2:=\Delta(y-x), ~~b_3:=\Delta((m-1)y-(n-1)x).
	\end{aligned}
\end{equation}
We need to ensure that  $B \geq \max\{|b_1|, |b_2|, |b_3|\}$, so we can take $B:=n^2$.
We also need to ensure that $A_i \geq \max\{Dh(\gamma_{i}), |\log\gamma_{i}|, 0.16\}$ for $ i=1,2,3 $. So, we take 
$A_1: = Dh(\gamma_{1}) = 2k\log k$. Since $h(\alpha)=(\log \alpha)/k$ and $h(2\alpha-1)<3/k$, we can take 
$A_2 := A_3 := 1$. 
Then, by Theorem \ref{thm:Mat},
\begin{align}\label{eq3.10}
	\log |\Gamma| &> -1.4\cdot 30^6 \cdot3^{4.5}\cdot k^2 (1+\log k)(1+2\log n)\cdot2k\log k\nonumber\\
	&> -1.6\cdot 10^{12}k^3(\log k)^2 \log n,
\end{align}
since $k\ge 3$ and $n>50$. Comparing \eqref{eq3.9} and \eqref{eq3.10}, we get
\begin{align}\label{eq3.11}
	m\log 1.5-\log 19m&<1.6\cdot 10^{12}k^3(\log k)^2 \log n
\end{align}
Next, we proceed in two cases, that is, we distinguish between the cases $n\le m^2$ and $n>m^2$.
\begin{enumerate}[(a)]
	\item When $n\le m^2$, we use \eqref{eq3.11} to obtain  that
\begin{align*}
m\log 1.5-\log 19m&<1.6\cdot 10^{12}k^3(\log k)^2 \log m^2;\\
	m\log 1.5&<3.2\cdot 10^{12}k^3(\log k)^2 \log m+\log 19+\log m;\\
	\dfrac{m}{\log m}&<8\cdot 10^{12}k^3(\log k)^2 
\end{align*}	
We now apply Lemma \ref{Guz} with the data $z:=m$, $s:=1$ and $T:=8\cdot 10^{12}k^3(\log k)^2 >(4s^2)^s=4$ since $k\ge 3$. We get 
\begin{align}\label{eq3.12}
m&<2\cdot 8\cdot 10^{12}k^3(\log k)^2\log [8\cdot 10^{12}k^3(\log k)^2]\nonumber\\
&=1.6\cdot 10^{13}k^3(\log k)^2\left(\log 8\cdot 10^{12}+3\log k+2\log\log k\right)\nonumber\\
&<1.6\cdot 10^{13}k^3(\log k)^2\left(30+5\log k\right)
 <5.3\cdot 10^{14}k^3 (\log k)^3.
\end{align}
So, we get that $n\le m^2<\left(5.3\cdot 10^{14}k^3 (\log k)^3\right)^2<2.9\times 10^{29}k^6 (\log k)^6$.
\item When $n>m^2$, we return to \eqref{eq3.2} and rewrite it using \eqref{eq3.7} as 
\begin{align*}
	\left|\left(L_m^{(k)}\right)^y-\left(f_k(\alpha)\right)^x(2\alpha-1)^x\alpha^{(n-1)x}\right|&< \dfrac{2n\left(f_k(\alpha)\right)^x(2\alpha-1)^x\alpha^{(n-1)x}}{1.5^{n-2}}.
\end{align*}
Dividing both sides of the above inequality by $\left(f_k(\alpha)\right)^x(2\alpha-1)^x\alpha^{(n-1)x}$, we get
\begin{align}\label{eq3.13}
	\left|\left(f_k(\alpha)\right)^{-x}(2\alpha-1)^{-x}\alpha^{-(n-1)x}\left(L_m^{(k)}\right)^y-1\right|&< \dfrac{4.5n}{1.5^{n}}<\dfrac{1}{1.5^{n/2}},
\end{align}
since $n>50$. 
We again apply Theorem \ref{thm:Mat} on the left--hand side of \eqref{eq3.13}. Here, we let $$\Gamma_1:=\left(f_k(\alpha)\right)^{-x}(2\alpha-1)^{-x}\alpha^{-(n-1)x}\left(L_m^{(k)}\right)^y-1=e^{\Lambda_1}-1.$$ Again, $\Lambda_1 \ne 0$, otherwise we would have $$\left(f_k(\alpha)\right)^{x}(2\alpha-1)^{x}\alpha^{(n-1)x}=\left(L_m^{(k)}\right)^y=\left(L_n^{(k)}\right)^x.$$
Upon simplifying the exponents, we get $f_k(\alpha)(2\alpha-1)\alpha^{n-1}=L_n^{(k)}$. Taking norms in ${\mathbb K}={\mathbb Q}(\alpha)$ 
and using that  $|N(\alpha)|=1$, by item (i) and (ii) of Lemma \ref{lemGL}, we get that
\begin{align*}
	1> N((2\alpha - 1)f_k(\alpha))=N(L_n^{(k)})=(L_n^{(k)})^k>1,
\end{align*}
a contradiction. Thus, $\Lambda_1\ne 0$. We again use the field $\mathbb{Q}(\alpha)$ with
degree $D = k$. Here, $t = 4$,
\begin{equation}\nonumber
	\begin{aligned}
		\gamma_{1}&:=f_k(\alpha), ~~\gamma_{2}:=2\alpha-1, ~~\gamma_{3}:=\alpha,~~\gamma_{4}:=L_m^{(k)},\\
		b_1&:=-x, ~~b_2:=-x, ~~b_3:=-(n-1)x ~~b_4=y.
	\end{aligned}
\end{equation}
We need $B \geq \max\{|b_1|, |b_2|, |b_3|, |b_4|\}$, so we can take $B:=n^2$. Like before, we take
$A_1  := 2k\log k$, and we can still take 
$A_2 := A_3: = 1$. For $A_4$, we note that $L_m^{(k)}\le 2\alpha^{m}$ by inequality \eqref{eq2.4}, so we can take $A_4 = 2mk$. Thus, by Theorem \ref{thm:Mat},
\begin{align}\label{eq3.14}
	\log |\Gamma| &> -1.4\cdot 30^7 \cdot 4^{4.5}\cdot k^2 (1+\log k)(1+2\log n)\cdot2k\log k\cdot 2mk\nonumber\\
	&> -3.5\cdot 10^{14}k^4(\log k)^2 \sqrt{n}\log n, ~~~\text{since}~m^2<n~~\text{in this case (b)}.
\end{align}
Comparing \eqref{eq3.13} and \eqref{eq3.14}, we get
\begin{align}\label{eq3.15}
	\dfrac{n}{2}\log 1.5&< 3.5\cdot 10^{14}k^4(\log k)^2 \sqrt{n}\log n; \nonumber\\
	\sqrt{n}&<1.8\cdot 10^{15}k^4(\log k)^2 \log n;\nonumber\\
	\dfrac{\sqrt{n}}{\log \sqrt{n}}&<3.6\cdot 10^{15}k^4(\log k)^2 .
\end{align}
Again, we now apply Lemma \ref{Guz} with $z:=\sqrt{n}$, $s:=1$ and $T:=3.6\cdot 10^{15}k^4(\log k)^2 >(4s^2)^s=4$ since $k\ge 3$. We get 
\begin{align}\label{eq3.16}
	\sqrt{n}&<2\times 3.6\cdot 10^{15}k^4(\log k)^2\log \left(3.6\cdot 10^{15}k^4(\log k)^2\right)\nonumber\\
	&=7.2\cdot 10^{15}k^4(\log k)^2\left(\log 3.6\cdot 10^{15}+4\log k+2\log\log k\right)\nonumber\\
	&<7.2\cdot 10^{15}k^4(\log k)^2\left(36+6\log k\right)\nonumber\\
	&<2.9\cdot 10^{17}k^4 (\log k)^3;\nonumber\\
	\text{or}~~~ n &< 8.5 \cdot10^{34}k^8 (\log k)^6.
\end{align}
\end{enumerate}
This completes the proof of Lemma \ref{lem3.1}.
\end{proof}

\subsection{Considerations on $k$}
In our analysis, we consider two distinct scenarios based on the value of $k$: when $k > 1000$ and when $k \leq 1000$. For the first scenario, we show that equation \eqref{eq3.2} has no solutions. Consequently, we focus on the case where $k \leq 1000$. However, it is noteworthy that Lemma \ref{lem3.1} imposes a different upper bound on $n$, specifically $n < 9.3\times 10^{63}$, which is exceedingly large and impractical for direct computation. To overcome this computational challenge, we reduce this upper bound in two cases, using inequalities analogous to \eqref{eq3.9} and \eqref{eq3.13}, involving the LLL--algorithm, as we shall elaborate.

\subsubsection{The case $k>1000$}
For this case, we have from Lemma \ref{lem3.1} that
\begin{align*}
	n < 8.5\cdot 10^{34}k^8 (\log k)^6<2^{2k/5},
\end{align*}
for all $k>1000$. Next, let $a$ and $b$ be nonnegative integers with $b>0$. We quote the following inequality which is proved as estimate (3.14) on page 229 of \cite{GGL}. It asserts that
\begin{align}\label{eq3.17}
	\left|\left(u_n^{(k)}\right)^x-2^{(n-2)x}(a+b)^x\right|<71\cdot \dfrac{2^{(n-2)x}(a+b)^x}{2^{2k/5}}.
\end{align}
For us here, $u_n^{(k)}:=L_n^{(k)}$, $a=L_0^{(k)}=2$ and $b=L_1^{(k)}=1$. Therefore, using inequality \eqref{eq3.17} together with relation \eqref{eq3.2}, we get
\begin{align}\label{eq3.18}
	\left|2^{(n-2)x}\cdot3^x-2^{(m-2)y}\cdot3^y\right|&=\left|2^{(n-2)x}\cdot3^x-2^{(m-2)y}\cdot3^y+\left(L_m^{(k)}\right)^y-\left(L_n^{(k)}\right)^x\right|\nonumber\\
	&\le\left|\left(L_n^{(k)}\right)^x-2^{(n-2)x}\cdot3^x\right|+\left|\left(L_m^{(k)}\right)^y-2^{(m-2)y}\cdot3^y\right|\nonumber\\
	&<71\cdot \dfrac{2^{(n-2)x}\cdot 3^x+2^{(m-2)y}\cdot 3^y}{2^{2k/5}}.
\end{align}
If we divide through by $\max\left\{2^{(n-2)x}\cdot3^x,~2^{(m-2)y}\cdot3^y\right\}$, we get
\begin{align*}
	\left|1-2^{\varepsilon((n-2)x-(m-2)y)}\cdot3^{\varepsilon(x-y)}\right|<\dfrac{142}{2^{2k/5}},
\end{align*}
where $\varepsilon\in \{\pm 1\}$. 
To proceed, let $\Gamma_2:=1-2^{\varepsilon((n-2)x-(m-2)y)}\cdot3^{\varepsilon(x-y)}$. If we write 
$$\Lambda_2:=\varepsilon((n-2)x-(m-2)y)\log 2+\varepsilon(x-y)\log 3,$$
then 
\begin{align}\label{eq3.19}
 |\Lambda_2|<e^{|\Lambda_2|}\left(e^{\Lambda_2}-1\right)<\dfrac{284}{2^{2k/5}},   
\end{align}
since $\left|e^{\Lambda_2}-1\right|=|\Gamma_2|<142/2^{2k/5}$ and $e^{|\Lambda_2|}\le1+|\Gamma_2|<2$ for all $k>1000$. We need to ensure that $\Gamma_2\ne 0$. But if $\Gamma_2=0$, we then get $x-y=0$ (from the exponent of $3$) and next $(n-2)x=(m-2)y$ (from the exponent of $2$), so since $x=y\ne 0$, we get $n-2=m-2$, or $n=m$, a contradiction. 

Now, we apply Theorem \ref{thm:LMN} on $\Lambda_2$ with $\gamma_1:=2$, $\gamma_{2}:=3$, $b_1:=\varepsilon((n-2)x-(m-2)y)$ and $b_2:=\varepsilon(x-y)$. Moreover, $\mathbb{K}:=\mathbb{Q}$, $D=1$ and so we can write $\log A_i:=2$, for $i=1,2$. Also, we have
$$
b':=\frac{|b_1|}{D\log A_2}+\frac{|b_2|}{D\log A_1}=\dfrac{1}{2}\left(|b_1|+|b_2|\right).
$$
Therefore,
\begin{align*}
\log |\Lambda_2|&\ge -24.34 \left(\max\left\{\log b'+0.14,~21,~0.5\right\}\right)^2\cdot 2\cdot 2\\
&>	-98 \left(\max\left\{\log b'+0.14,~21\right\}\right)^2.
\end{align*}
Using the upper bound from relation \eqref{eq3.19}, we get 
\begin{align*}
	\frac{2}{5}k\log 2-\log 284&<98 \left(\max\left\{\log b'+0.14,~21\right\}\right)^2,\\
	k&<354 \left(\max\left\{\log b'+0.14,~21\right\}\right)^2 + 21.
\end{align*}
Now, if $\max\left\{\log b'+0.14,~21\right\}=21$, then $k<156135$ and if $\max\left\{\log b'+0.14,~21\right\}=\log b'+0.14$, then 
\begin{align*}
	b'&=\dfrac{1}{2}\left(|b_1|+|b_2|\right)= \dfrac{1}{2}\left(|((n-2)x-(m-2)y)|+|(x-y)|\right)\\
	&\le \dfrac{n+n^2}{2}<n^2<7.3\cdot 10^{69}k^{16}(\log k)^{12},
\end{align*}
where we have used Lemma \ref{lem3.1}. Hence, 
\begin{align*}
	\max\left\{\log b'+0.14,~21\right\}&=\log b'+0.14\\
	&<\log \left(7.3\cdot 10^{69}k^{16}(\log k)^{12}\right)+0.14\\
	&=\log (7.3\cdot 10^{69})+16\log k+12\log\log k+0.14\\
&<162+28\log k<52\log k, \quad \text{for}\quad k>1000.
\end{align*}
We thus get that $k<354\cdot (52\log k)^2+21< 10^6(\log k)^2$. Applying Lemma \ref{Guz} with $z:=k$, $s:=2$ and $T:=10^{6} >(4s^2)^s=256$, we get $k<7.7\cdot 10^8$. Therefore, in all cases, $k<7.7\cdot 10^8$. With this upper bound, we have by Lemma \ref{lem3.1} that $n<7.8\cdot 10^{113}$.

Next, we divide \eqref{eq3.19} by $(x-y)\log 2$, we have
\begin{align}
	\left|\dfrac{\log 3}{\log 2} - \dfrac{(n-2)x-(m-2)y}{x-y} \right|<\dfrac{410}{2^{2k/5}(x-y)},
\end{align}
since $x$ and $y$ are distinct. By Lemma \ref{lem:Leg} with $\tau:=\dfrac{\log 3}{\log 2} $ and $M:=10^{114}$, we have 
\begin{align*}
	\dfrac{1}{(a(M)+2)(x-y)^2}<\left|\dfrac{\log 3}{\log 2} - \dfrac{(n-2)x-(m-2)y}{x-y} \right|<\dfrac{410}{2^{2k/5}(x-y)},
\end{align*}
where $a(M)=100$ (in fact, $q_{229}>10^{114}$ and $\max\{a_k: 0\le k\le 229\}=a_{218}=100$). The above inequality gives
\begin{align*}
	\dfrac{1}{(100+2)(x-y)^2}&<\dfrac{410}{2^{2k/5}(x-y)};\\
	2^{2k/5}&<41820(x-y)<41820n;\\
	&<3.6\cdot 10^{39}k^8(\log k)^6. 
\end{align*}
Taking logarithms both sides and simplifying, we get
\begin{align*}
	\dfrac{2}{5}k\log 2<\log (3.6\cdot 10^{39})+8\log k+6\log\log k<92+14\log k<28\log k,
\end{align*}
which implies that $k<101\log k$. Applying Lemma \ref{Guz} with $z:=k$, $s:=1$ and $T:=101 >(4s^2)^s=4$, we get $k<933$. This contradicts the working assumption $k>1000$.

\subsubsection{The case $k\le1000$}
In this case, since  $k\le 1000$, then $n<8.5 \cdot10^{34}(1000)^8 (\log 1000)^6<9.3\cdot 10^{63}$, by Lemma \ref{lem3.1}. The purpose here is to reduce this upper bound on $n$. To do this, we first go back to \eqref{eq3.9} and recall that we had already obtained 
	$$\Gamma:=\left(f_k(\alpha)\right)^{\Delta(y-x)}(2\alpha-1)^{\Delta(y-x)}\alpha^{\Delta((m-1)y-(n-1)x)}-1=e^{\Lambda}-1.$$
We already showed that $\Lambda\ne 0$, so $\Gamma\ne 0$. Assume for a moment that $m>14$, then we have that $\left|e^{\Lambda}-1\right|=|\Gamma|<1$ and $e^{|\Lambda|}\le1+|\Gamma|<2$, implying that 
\begin{align*}
	|(m-1)y-(n-1)x|&<\dfrac{1}{\log \alpha}\left| 2+|y-x|\log(2\alpha-1)+|y-x||\log f_k(\alpha)|  \right|\\
	&<\dfrac{1}{\log\alpha}(2+2n+n)=\dfrac{n}{\log\alpha}\cdot \left(3+\dfrac{2}{n}\right)<9n,
\end{align*}
since $m>14$. In the calculation above, we have used $|\log(2\alpha-1)|<2$, $|\log f_k(\alpha)|<1$ and $n:=\max\{x,y\}$. We can therefore rewrite \eqref{eq3.9} as
\begin{align*}
	|\Lambda|<e^{|\Lambda|}|e^{\Lambda}-1|<\dfrac{38m}{1.5^{m}}<\dfrac{10^{28}}{1.5^{m}},
\end{align*}
where we have used \eqref{eq3.12} and $k\le 1000$. 

Now, for each $k \in [3, 1000]$, we use the LLL--algorithm to compute a lower bound for the smallest nonzero number of the form 
$$|\Lambda|:=\left|(y-x)\log f_k(\alpha)+(y-x)\log(2\alpha-1)+((m-1)y-(n-1)x)\log\alpha\right|,$$
 with integer coefficients not exceeding $9n<8.4\cdot 10^{64}$ in absolute value. Specifically, we consider the approximation lattice
$$ \mathcal{A}=\begin{pmatrix}
	1 & 0 & 0 \\
	0 & 1 & 0 \\
	\lfloor C\log f_k(\alpha)\rfloor & \lfloor C\log(2\alpha-1)\rfloor& \lfloor C\log\alpha \rfloor
\end{pmatrix} ,$$
with $C:= 6\cdot 10^{194}$ and choose $y:=\left(0,0,0\right)$. Now, by Lemma \ref{lem2.5}, we get $$l\left(\mathcal{L},y\right)=|\Lambda|>c_1=10^{-68}\qquad\text{and}\qquad\delta=10^{66}.$$
So, Lemma \ref{lem2.6} gives $S=2.22\cdot 10^{130}$ and $T=1.26\cdot 10^{65}$. Since $\delta^2\ge T^2+S$, then choosing $c_3:=10^{28}$ and $c_4:=\log 1.5$, we get $m<891$.

Recall that we are trying to reduce the upper bound on $n$. At this point with a reduced bound on $m$, we proceed in two cases.
\begin{enumerate}[(a)]
\item When $n\le m^2$, we have from the reduced bound on $m$ that $n<891^2=793881$. With this new upper bound for $n$ we repeat the LLL--algorithm to get a lower bound of $|\Lambda|$, where now the coefficients of the linear form $|\Lambda|$ are integers not exceeding $9n<7.2\cdot 10^6$.
With the same approximation lattice and $C:=10^{21}$, we get $|\Lambda|>10^{-16}$, $S=2.54\cdot 10^{14}$ and $T=1.39\cdot 10^{7}$. Moreover, choosing $c_3:=38m=33858$ and $c_4:=\log 1.5$, we obtain that $m < 111$, or similarly $n\le m^2<12321$. Repeating this algorithm twice, we conclude that $n<1600$.

\item When $n>m^2$, then we go back to \eqref{eq3.13} where we deduced
$$\Gamma_1:=\left(f_k(\alpha)\right)^{-x}(2\alpha-1)^{-x}\alpha^{-(n-1)x}\left(L_m^{(k)}\right)^y-1=e^{\Lambda_1}-1.$$
Clearly, $\Gamma_1\ne0$ since we already showed that $\Lambda_1\ne 0$. If $\Gamma_1 > 0$, then $e^{\Lambda_1} - 1 > 0$, so from \eqref{eq3.13} we obtain
\[
0 < \Gamma_1 <  \frac{1}{1.5^{n/2}}.
\]
If $\Gamma_1 < 0$, then we can note from \eqref{eq3.13} that $1/1.5^{n/2} < 1/2$ since $n>50$. Hence, it follows that $|e^{\Lambda_1} - 1| < 1/2$ which implies $e^{\Lambda_1} < 2$. Since $\Gamma_1 < 0$, we obtain that
\begin{equation*}
	0 < |\Lambda_1| \leq e^{|\Lambda_1|} - 1 = e^{|\Lambda_1|}|e^{\Lambda_1} - 1| < \frac{2}{1.5^{n/2}}.
\end{equation*}
Thus, in all cases, we have 
\begin{equation}
	|\Lambda_1| < \frac{2}{1.5^{n/2}}.
\end{equation}
Observe that $|\Lambda_1|$ is an expression of the form
\[
|a_1 \log f_k(\alpha) + a_2 \log (2\alpha-1) + a_3 \log \alpha + a_4 \log L_m^{(k)}|,
\]
where $a_1 := -x$, $a_2 := -x$, $a_3 := -(n-1)x$, $a_4 := y$. From Lemma \ref{lem3.1}, we have
\[
\max\{|a_i| : 1 \leq i \leq 4\} < n^2 <  10^{128}.
\]
At this point when $k\in[3,1000]$ and $m\in[3,890]$, we consider the approximation lattice
$$ \mathcal{A}=\begin{pmatrix}
	1 & 0 & 0 & 0  \\
	0 & 1 & 0 & 0  \\
	0 & 0 & 1 & 0  \\
	
	\lfloor C\log f_k(\alpha)\rfloor & \lfloor C\log (2\alpha-1)\rfloor& \lfloor C\log \alpha \rfloor& \lfloor C\log L_m^{(k)} \rfloor 
\end{pmatrix} ,$$
with $C:= 10^{512}$ and choose $y:=\left(0,0,0,0\right)$. By Lemma \ref{lem2.5}, we get $|\Lambda_1|>c_1=10^{-130}$ and $\delta=10^{129}$. So, Lemma \ref{lem2.6} gives $S=2\cdot 10^{256}$ and $T=1.5\cdot 10^{128}$. Since $\delta^2\ge T^2+S$, then choosing $c_3:=2$ and $c_4:=\log 1.5$, we get $n/2<2178$. This implies that $n\le 4356$.

With this new upper bound for $n$ we repeat the LLL--algorithm once again to get a lower bound of $|\Lambda_1|$, where now the coefficients $a_i$ are integers satisfying 
\[
\max\{|a_i|: 1 \leq i \leq 4\} < n^2 < 1.9 \cdot 10^{7}.
\]
With the same approximation lattice and $C:=10^{30}$, we get $|\Lambda_1|>10^{-10}$, $\delta=10^8$ $S=1.45\cdot 10^{15}$ and $T=3.81\cdot 10^{7}$. We then obtain that $n \leq 258$.
\end{enumerate}
So, in all cases $n<1600$. Finally, we do a computational search for solutions to our Diophantine equation \eqref{eq3.2}, when $3 \leq k \leq 1000$ and $3 <m< n < 1600$. First, for $k$ fixed, we use SageMath to calculate 
\[ M(n, m) := \gcd(L^{(k)}_n, L^{(k)}_m) \]
and we print the pairs $(n, m)$, with $n, m \in [3, 1600]$, such that 
\[ \text{power\textunderscore mod}(M(n, m), n, L^{(k)}_n) = 0 \quad \text{and} \quad \text{power\textunderscore mod}(M(n, m), m, L^{(k)}_m) = 0. \]
Here, $\text{power\textunderscore mod}(A, r, B)$ calculate $A^r \pmod{B}$. To conclude, we verify which of these pairs correspond to solutions of equation \eqref{eq3.2}, for each $k \in [3, 1000]$. We find none, see Appendix 3. This program ran for about 20 hours.

\section*{Acknowledgments} 
The first author thanks the Eastern Africa Universities Mathematics Programme (EAUMP) for funding his doctoral studies. The first and last co-authors were supported in part by Grant \#2024-029-NUM from Wits, Johannesburg, South Africa. The last author worked on this paper during a fellowship at STIAS. This author thanks STIAS for hospitality and support.

\section*{Addresses}
$ ^{1} $ Department of Mathematics, School of Physical Sciences, College of Natural Sciences, Makerere University, Kampala, Uganda

 Email: \url{hbatte91@gmail.com}
 
 Email: \url{mahadi.ddamulira@mak.ac.ug}
 
 Email: \url{juma.kasozi@mak.ac.ug}

\vspace{0.35cm}
\noindent 
$ ^{2} $ School of Mathematics, Wits University, Johannesburg, South Africa and Centro de Ciencias Matem\'aticas UNAM, Morelia, Mexico 

Email: \url{Florian.Luca@wits.ac.za}
\newpage
\section*{Appendices}
\subsection*{Appendix 1}\label{app1}
\begin{verbatim}
memo = {}
def k_generalized_lucas(n, k):
    if (n, k) in memo:
        return memo[(n, k)]
    # Base cases
    if n < 2-k:
        return 0
    elif n == 0:
        return 2
    elif n == 1:
        return 1
    else:
         # Recurrence relation
         result = sum([k_generalized_lucas(n-i, k) for i in range(1, k+1)])
         memo[(n, k)] = result
         return result
def check_multiplicative_dependence(a, b):
    # Check for multiplicative dependence
    for x in range(1, n):  
       for y in range(1, n):
            if x != y and a ^ x == b ^ y:
                 return True
    return False
# Check for multiplicative dependence
for k in range(2, 51):
   for n in range(2, 51):
       for m in range(2, n):  # Ensure m < n
           lnk = k_generalized_lucas(n, k)
           lmk = k_generalized_lucas(m, k)
           if check_multiplicative_dependence(lnk, lmk):
              print(f"For k={k}, the numbers L_{n}^{(k)}={lnk} and L_{m}^{(k)}={lmk} 
              are multiplicatively dependent.")
\end{verbatim}
\subsection*{Appendix 2}\label{app2}
\begin{verbatim}
tau = log(3)/log(2)
M = 10^114
cf_tau = continued_fraction(tau)
N = 0

for convergent in cf_tau.convergents():
    if convergent.denominator() > M:
       break
    N += 1
    
print(f"The index N such that q_N > M is: {N}")
print(f"The convergent p_N/q_N is: {cf_tau.convergents()[N-1]}")

a_M = max(cf_tau[:N+1])
print(f"The value a(M) is: {a_M}")
	
\end{verbatim}
\subsection*{Appendix 3}\label{app3}
\begin{verbatim}
def k_generalized_lucas(n, k, memo={}):
    if (n, k) in memo:
       return memo[(n, k)]
    if n < 2 - k:
       return 0
    elif n == 0:
       return 2
    elif n == 1:
       return 1
    else:
       result = sum([k_generalized_lucas(n - i, k, memo) for i in range(1, k + 1)])
       memo[(n, k)] = result
       return result
# Function to check the GCD conditions
def check_gcd_conditions(n, m, k):
    ln = k_generalized_lucas(n, k)
    lm = k_generalized_lucas(m, k)
    gcd_val = gcd(ln, lm)
    if power_mod(gcd_val, n, ln) == 0 and power_mod(gcd_val, m, lm) == 0:
        return True
    return False
# Function to check the multiplicative dependence
def check_multiplicative_dependence(n, m, k):
    ln = k_generalized_lucas(n, k)
    lm = k_generalized_lucas(m, k)
    for x in range(1, n):
       for y in range(1, n):
           if x != y and ln ^ x == lm ^ y:
               return True
    return False
pairs = []
for k in range(3, 1001):
    for n in range(4, 1600):
        for m in range(3, n):
            if check_gcd_conditions(n, m, k):
                  pairs.append((k, n, m))
for k, n, m in pairs:
     if check_multiplicative_dependence(n, m, k):
          print(f"For k={k}, the pair (n={n}, m={m}) satisfies the 
          multiplicative dependence condition.")
\end{verbatim}
\newpage
\subsection*{Appendix 4}\label{app4}
\begin{verbatim}
from sage.all import *
# Define alpha and C
alpha = (1 + sqrt(5)) / 2
C = 6e194  # 6 * 10^194
	
# Define f_k(x) function
def f_k(x, k):
    return (x - 1) / (2 + (k + 1) * (x - 2))
# Initialize variables to store the best c1, delta, and corresponding k
best_c1 = float('inf')  # Initialize to a large value
best_delta = 0
best_k = None
	
# Loop over k values
for k in range(3, 1001):  # Adjust the range as needed
    # Define matrix B for each k
    B = Matrix(QQ, [[1, 0, 0], 
                   [0, 1, 0], 
                   [floor(C * (log(f_k(alpha, k)).n())), 
                    floor(C * (log(2*alpha - 1).n())), 
                     floor(C * (log(alpha).n()))]])
	
    # Apply LLL reduction to get a reduced basis
    B_reduced = B.LLL()
	
    # Perform Gram-Schmidt orthogonalization on the reduced basis
    B_orthogonalized = B_reduced.gram_schmidt()[0]
	
    # Compute c1
    c1 = min([B_reduced.column(0).norm() / B_orthogonalized[i].norm() 
    for i in range(B_reduced.ncols())])
	
    # Define y and the lattice L
    y = vector([0, 0, 0])  # Replace with your actual vector
    L = B_reduced * ZZ^B_reduced.ncols()
	
    # Compute lambda and delta
    if y not in L:
        z = B_reduced.inverse() * y
        i0 = max([i for i in range(len(z)) if z[i] != 0])
        lambda_val = abs(z[i0] - round(z[i0]))
    else:
        lambda_val = 1
	
    delta = lambda_val * B_reduced.column(0).norm() / c1
	
    # Update best_c1, best_delta, and best_k if a better c1 is found
    if c1 < best_c1:
       best_c1 = c1
       best_delta = delta
       best_k = k
	
# Output the best c1, delta, and corresponding k
print(f"Best k: {best_k}")
print(f"Best c1: {float(best_c1):.6e}")
print(f"Best delta: {float(best_delta):.6e}")
\end{verbatim}

\end{document}